\documentclass[10pt,reqno]{amsart}

\usepackage{amssymb}
\usepackage{amsmath}
\usepackage{amsfonts}
\usepackage{amscd}
\usepackage{mathrsfs}
\usepackage{color}

\newtheorem{thm}{Theorem}[section]

\newtheorem{lem}[thm]{Lemma}
\newtheorem{prop}[thm]{Proposition}

\theoremstyle{definition}

\theoremstyle{remark}

\theoremstyle{remark}

\numberwithin{equation}{section}

\newcommand{\bdy}{\partial\Omega}
\newcommand{\BR}{\mathbf{B}_R}
\newcommand{\BRc}{\mathbf{B}^c_R}
\newcommand{\Jc}{\mathcal{J}_{\lambda}}
\newcommand{\Jclk}{\mathcal{J}_{\lambda(k)}}
\newcommand{\Ms}{\mathscr{M}}
\newcommand{\Om}{\Omega}
\newcommand{\Ov}{\Omega_\varepsilon}
\newcommand{\RN}{\mathbf{R}^N}
\newcommand{\R}{\mathbf{R}}
\newcommand{\La}{-\Delta\hspace{0.2mm}}
\newcommand{\Lp}{-\Delta_p\hspace{0.2mm}}

\newcommand{\bre}[1]{\left\{#1\right\}}

\newcommand{\n}[1]{\left\vert#1\right\vert}
\newcommand{\nm}[1]{\left\Vert#1\right\Vert}

\newcommand{\dx}{\hspace{0.36mm}dx}
\newcommand{\gradu}{\nabla u}
\newcommand{\gradul}{\nabla\ul}
\newcommand{\gradulk}{\nabla\ulk}
\newcommand{\gradum}{\nabla\um}
\newcommand{\gradv}{\nabla v}
\newcommand{\gradvl}{\nabla\vl}
\newcommand{\gradw}{\nabla\omega}
\newcommand{\gradx}{\nabla\xi}

\newcommand{\IOm}{\int_{\Om}}
\newcommand{\IOv}{\int_{\Ov}}
\newcommand{\IR}{\int_{\RN}}
\newcommand{\IBR}{\int_{\mathbf{B}_R}}
\newcommand{\IBRc}{\int_{\mathbf{B}^c_R}}
\newcommand{\ul}{u_\lambda}
\newcommand{\ult}{\tilde{u}_\lambda}
\newcommand{\ulk}{u_{\lambda(k)}}
\newcommand{\um}{u_\mu}
\newcommand{\ve}{v_\varepsilon}
\newcommand{\vl}{v_\lambda}
\newcommand{\vv}{\varphi_\varepsilon}

\newcommand{\CCON}{C^1_c\left(\RN\right)}
\newcommand{\DpR}{D^{1,\hspace{0.2mm}p}\left(\RN\right)}
\newcommand{\LalocR}{L^{\alpha}_{loc}\left(\RN\right)}
\newcommand{\LaOm}{L^{\alpha}\left(\Om\right)}
\newcommand{\LaR}{L^{\alpha}\left(\RN\right)}
\newcommand{\LlocR}{L^1_{loc}\left(\RN\right)}
\newcommand{\LlR}{L^1\left(\RN\right)}
\newcommand{\LpR}{L^p\left(\RN\right)}
\newcommand{\LpOm}{L^p\left(\Om\right)}
\newcommand{\LpcR}{L^{p^*}\left(\RN\right)}
\newcommand{\LpccR}{L^{\frac{p^*}{p^*-r}}\left(\RN\right)}
\newcommand{\LqR}{L^q\left(\RN\right)}
\newcommand{\LqOm}{L^q\left(\Om\right)}
\newcommand{\LVR}{L^q_{\mathrm V}\left(\RN\right)}
\newcommand{\LVOv}{L^q_{\mathrm V}\left(\Ov\right)}

\newcommand{\LKR}{L^r_{\mathrm K}\left(\RN\right)}
\newcommand{\LKBR}{L^r_{\mathrm K}\left(\BR\right)}
\newcommand{\LKOv}{L^r_{\mathrm K}\left(\Ov\right)}
\newcommand{\LsR}{L^s\left(\RN\right)}
\newcommand{\LsOm}{L^s\left(\Om\right)}
\newcommand{\LtOm}{L^t\left(\Om\right)}
\newcommand{\MOm}{M^{q,\hspace{0.2mm}p}\left(\Om\right)}
\newcommand{\MBR}{M^{q,\hspace{0.2mm}p}\left(\BR\right)}
\newcommand{\MRN}{M^{q,\hspace{0.2mm}p}\left(\RN\right)}
\newcommand{\McR}{M^{p^*,\hspace{0.2mm}p}\left(\RN\right)}
\newcommand{\MVBR}{M^{q,\hspace{0.2mm}p}_{\mathrm V}\left(\BR\right)}
\newcommand{\MVR}{M^{q,\hspace{0.2mm}p}_{\mathrm V}\left(\RN\right)}
\newcommand{\MqoOm}{M^{q_1,\hspace{0.2mm}p}\left(\Om\right)}
\newcommand{\MqtOm}{M^{q_2,\hspace{0.2mm}p}\left(\Om\right)}
\newcommand{\MinOm}{M^{\infty,\hspace{0.2mm}p}\left(\Om\right)}
\newcommand{\MinR}{M^{\infty,\hspace{0.2mm}p}\left(\RN\right)}
\newcommand{\MqoR}{M^{q_1,\hspace{0.2mm}p}\left(\RN\right)}
\newcommand{\MqtR}{M^{q_2,\hspace{0.2mm}p}\left(\RN\right)}
\newcommand{\Wop}{W^{1,\hspace{0.2mm}p}\left(\Om\right)}
\newcommand{\WpR}{W^{1,\hspace{0.2mm}p}\left(\RN\right)}
\newcommand{\Wlpq}{W^1_{p,\hspace{0.2mm}q}\left(\RN\right)}
\begin{document}

\title[Compact Sobolev embeddings and positive solutions]
{\bf Compact Sobolev embeddings and positive solutions to a quasilinear equation with indefinite nonlinearities}

\author[Q. Han]
{\sf Qi Han}

\address{Department of Science and Mathematics, Texas A\&M University at San Antonio
\vskip 2pt San Antonio, Texas 78224, USA \hspace{14.4mm}{\sf Email: qhan@tamusa.edu}}

%\email{qhan@tamusa.edu}
%\thanks{{\sf Dedicated to my little angel Jacquelyn and her mother, my dear wife, Jingbo.}}
%\date{}

\dedicatory{Dedicated to my little angel Jacquelyn and her mother, my dear wife, Jingbo.}

%\commby{}

% -----------------------------------------------------------------------------

\begin{abstract}
  In this paper, we study the existence and multiplicity results of nontrivial positive solutions to the following quasilinear elliptic equation on $\RN$, when $N\geq2$,
  \begin{equation}
  \Lp u=\lambda\hspace{0.2mm}K(x)u^{r-1}-V(x)u^{q-1}.\nonumber
  \end{equation}
  Here, $K(x),V(x)>0$ are suitable potentials, $1<p<r<q<\infty$, and $\lambda>0$ is a parameter.
  To study this problem, some compact embedding results regarding $\MVR\hookrightarrow\LKR$ are proved that unify and extend some recent results of the author \cite{Ha1,Ha2,Ha3}.
\end{abstract}

% -----------------------------------------------------------------------------
\maketitle
% -----------------------------------------------------------------------------

% +++ below is section 1 %%%%%%%%%%%%%%%%%%%%%%%%%%%%%%%%%%%%%%%%%%%%%%%%%%%%%%%%%%%%%%%%%%%%%%%%%%%%%%%%%%%%%%%%%%%%%%%%%%%%%%%%%%%%%%%%%%%%%%%%%%%
% &&& %%%%%%%%%%%%%%%%%%%%%%%%%%%%%%%%%%%%%%%%%%%%%%%%%%%%%%%%%%%%%%%%%%%%%%%%%%%%%%%%%%%%%%%%%%%%%%%%%%%%%%%%%%%%%%%%%%%%%%%%%%%%%%%%%%%%%%%%%%%%%%
\section{Introduction}\label{Int} % use lowercase except for proper names
\noindent Ambrosetti, Brezis and Cerami in the seminal paper \cite{ABC} studied the existence and multiplicity results of positive solutions to the following elliptic equation
\begin{equation}
\La u=\lambda\hspace{0.2mm}u^{\rho-1}+u^{\varrho-1}\nonumber
\end{equation}
on a bounded, smooth domain subject to the Dirichlet data $u=0$, where $1<\rho<2<\varrho<2^*$ and $\lambda>0$ is a constant; see the fine paper of Ambrosetti, Garcia-Azorero and Peral \cite{AGP} as well.
There have been quite a few papers devoted to the study of similar problems, and in particular Alama and Tarantello \cite{AT} investigated a related problem with indefinite nonlinearities
\begin{equation}\label{eq1.1}
\La u-\lambda\hspace{0.2mm}u=k(x)u^{\rho-1}-h(x)u^{\varrho-1},\hspace{2mm}u>0
\end{equation}
in the same context as assumed in \cite{ABC,AGP} with $2<\rho<\varrho$ and some constant $\lambda\in\R$.
Notice the existence results for equation \eqref{eq1.1} depend on both the values of $\lambda$ and the integrability of the ratio function $k^{\alpha_1}(x)/h^{\alpha_2}(x)$ for some explicit exponents $\alpha_1,\alpha_2$ concerning $N,\rho,\varrho$.

Chabrowski \cite{Ch}, and Pucci and R\u{a}dulescu \cite{PR,Ra} recently extended the above work of Alama and Tarantello to $\RN$, and they studied the following quasilinear elliptic equation
\begin{equation}
\Lp u+u^{p-1}=\lambda\hspace{0.2mm}u^{r-1}-h(x)u^{q-1},\hspace{2mm}u\geq0.\nonumber
\end{equation}
Here, $h(x)>0:\RN\to\R$ satisfies an integrability condition, $2\leq p<r<q<2^*$, and $\lambda>0$ is a constant.
Existence, nonexistence and multiplicity results are given in \cite{Ch,PR,Ra}.
One may also check the interesting papers \cite{CD,AP,RS} and the references therein for related results.

In this paper, we are concerned with the existence of nontrivial positive solutions to
\begin{equation}\label{eq1.2}
\Lp u=\lambda\hspace{0.2mm}K(x)u^{r-1}-V(x)u^{q-1}
\end{equation}
on $\RN$ when $N\geq2$ in the function space $\MVR$.
Here, we assume that $K(x),V(x)>0:\RN\to\R$ are appropriate potentials, $1<p<r<q<\infty$, and $\lambda>0$ is a parameter.

As shown in \cite{Ra}, problem \eqref{eq1.2} is related to the Lane-Emden-Fowler equation that arises in the boundary-layer theory of viscous fluids; see for example the survey \cite{Wo}.
This problem goes back to the work of Lane in 1869 and was originally motivated by his interest in computing both the temperature and the density of mass on the surface of the sun; equation \eqref{eq1.2} characterizes the behavior of the density of a gas sphere in the hydrostatic equilibrium, where the index $r$ (the polytropic index in astrophysics) is related to the ratio of the specific heats of the gas.
On the other hand, as claimed in \cite{PR}, \eqref{eq1.2} may be viewed as a pattern formation prototype in biology associated with the steady-state problem modelling chemotactic aggregation, as introduced by Keller and Segel \cite{KS}; it also plays an important role in the study of activator-inhibitor systems modelling biological pattern formation, as proposed by Gierer and Meihardt \cite{GM}.
Other aspects of applications regarding problem \eqref{eq1.2} and a number of recent general application results can be found for instance in the monograph of Ghergu and R\u{a}dulescu \cite{GR}.

We now describe suitable function space settings for our work.
Let $V(x)>0$ be a Lebesgue measurable function in $\RN$.
When $1\leq p<N$, we designate $\DpR$, the space of functions $u$ with $u\in\LpcR$ and $\n{\gradu}\in\LpR$, as the base space to define
\begin{equation}
\MVR=\bre{u\in\DpR:\nm{u}_{\MVR}=\nm{u}_{\LVR}+\nm{\n{\gradu}}_{p,\hspace{0.2mm}\RN}<\infty},\nonumber
\end{equation}
where $p^*=\frac{Np}{N-p}$ denotes the Sobolev critical index and for $1\leq q<\infty$ we write
\begin{equation}
\LVR=\bre{u\in\LlocR:\nm{u}^q_{\LVR}=\IR\n{u}^qV\dx<\infty}.\nonumber
\end{equation}
When $N\leq p<\infty$, we assume $\inf\limits_{\RN}V(x)\geq V_0>0$ and write $\LVR$ just out of $\LqR$ for $1\leq q<\infty$, and then define $\MVR$ further requesting $\n{\gradu}\in\LpR$.

When $V(x)\equiv1$, one has the space $\MRN$ as analyzed in Han \cite{Ha1,Ha2,Ha3} that may be viewed as a natural extension of the classical Sobolev space $\WpR$.
Note $\MRN$ was initially introduced in Maz'ya \cite[section 5.1.1]{Ma} using the notation $\Wlpq$.

\vskip 9pt
{\bf Standing Assumptions.}
\vskip 6pt
\noindent{\bf(i).} $N\geq2$, $p\in\left(1,\infty\right)$, $q\in\left(p,\infty\right)$, $r\in\left(p,q\right)$, and $\lambda>0$.
\vskip 6pt
\noindent{\bf(ii).} $K(x),V(x)>0$ are Lebesgue measurable functions with $K(x)\in\LalocR$ for some $\alpha\in\left(\frac{\max\left\{p^*,q\right\}}{\max\left\{p^*,q\right\}-r},\infty\right]$ and $\inf\limits_{\mathbf{D}}V(x)\geq V_{\mathbf{D}}>0$ for each compact subset $\mathbf{D}\Subset\RN$ when $1<p<N$ while $K(x)\in\LalocR$ for some $\alpha\in\left(1,\infty\right]$ and $\inf\limits_{\RN}V(x)\geq V_0>0$ when $N\leq p<\infty$.
\vskip 0pt
\noindent{\bf(iii).} When $1<p<N$, $K^{\frac{p^*+\beta\left(p^*-q\right)}{p^*-r}}(x)V^{-\beta}(x)\in\LlR$ if $p<r<\min\left\{p^*,q\right\}$ for some $\beta\in\left(\frac{p^*\left(r-p\right)}{\left(p^*-p\right)\left(q-r\right)},\frac{r}{q-r}\right]$ and $K^{\frac{q}{q-r}}(x)V^{-\frac{r}{q-r}}(x)\in\LlR$ if $p^*\leq r<q$; when $N\leq p<\infty$, $K^{\frac{s+\beta\left(s-q\right)}{s-r}}(x)V^{-\beta}(x)\in\LlR$ if $p<r<q$ for some $\beta\in\left(\max\left\{0,\frac{Nq+pr-Np-pq}{p\left(q-r\right)}\right\},\frac{r}{q-r}\right]$ and $s\in\left[q,\infty\right)$.
\vskip 0pt
\noindent{\bf(iv).} $p\in\left(1,N\right)$, $q\in\left(p,\infty\right)$, $r\in\left(p,\min\left\{p^*,q\right\}\right)$, and $K(x)\in\LpccR$.
\vskip 9pt

\begin{thm}\label{T1.1}
Under the {\bf Standing Assumptions (i)-(iii)}, there exists a $\lambda_1\geq0$ such that equation \eqref{eq1.2} has at least a positive solution in $\MVR$ provided $\lambda>\lambda_1$.
Furthermore, if the {\bf Standing Assumption (iv)} holds, then $\lambda_1>0$ and equation \eqref{eq1.2} has at least a positive solution in $\MVR$ if and only if $\lambda\geq\lambda_1$; moreover, there is a $\lambda_2\left(\geq\lambda_1\right)$ such that equation \eqref{eq1.2} has at least two positive solutions in $\MVR$ for every $\lambda>\lambda_2$.
\end{thm}

%Our result extends the works of \cite{AP,Ch,PR,Ra} in several directions: for example, we consider $1<p<\infty$ but not merely $1<p<N$, and $p<r<q$ but not merely $p<r<\min\left\{p^*,q\right\}$ when $p\in\left(1,N\right)$; besides, we do not require $r>2$ and our hypotheses {\bf(ii)-(iii)}, especially the latter, are broader than those applied in \cite{AP}.
%Yet the author does appreciate and highly recommends the paper \cite{AP} where many issues were made mathematically rigorous, especially theorem a.3 in \cite{AP}; see also theorem 2.5 from Candela and Palmieri \cite{CP}.
%These mountain pass theorems are of independent interest and very helpful in attacking various problems arising from other settings.
%Our strategy of proof exploits similar ideas to those used in \cite{AP,Ch,PR,Ra}.

% +++ below is section 2 %%%%%%%%%%%%%%%%%%%%%%%%%%%%%%%%%%%%%%%%%%%%%%%%%%%%%%%%%%%%%%%%%%%%%%%%%%%%%%%%%%%%%%%%%%%%%%%%%%%%%%%%%%%%%%%%%%%%%%%%%%%
% &&& %%%%%%%%%%%%%%%%%%%%%%%%%%%%%%%%%%%%%%%%%%%%%%%%%%%%%%%%%%%%%%%%%%%%%%%%%%%%%%%%%%%%%%%%%%%%%%%%%%%%%%%%%%%%%%%%%%%%%%%%%%%%%%%%%%%%%%%%%%%%%%
\section{Function space deliberations}\label{FSD} % use lowercase except for proper names
\noindent This section is devoted to the analyses of the function space settings that will be needed later.
From now on, we shall denote both continuous embedding of function spaces and convergence of functions by `` $\to$ '', compact embedding of function spaces by `` $\hookrightarrow$ '', and weak convergence of functions by `` $\rightharpoonup$ ''.
Other notations will be specified when appropriate.

Recall $\MOm$ is described as the space of Sobolev functions $u$ on $\Om$ that are in $\LqOm$ but $\n{\gradu}$ are in $\LpOm$ for $1\leq p,q\leq\infty$.
It is a Banach space with respect to the norm
\begin{equation}\label{eq2.1}
\nm{u}_{\MOm}=\nm{u}_{q,\hspace{0.2mm}\Om}+\nm{\n{\gradu}}_{p,\hspace{0.2mm}\Om}.
\end{equation}

First, consider the case where $\Om$ is a bounded domain with a compact, Lipschitz boundary $\bdy$.
When $p\in\left[1,N\right)$, we have
\begin{equation}
\left\{\begin{array}{ll}
\MOm=\Wop&\mathrm{for}\hspace{2mm}1\leq q\leq p^*,\\\\
\MqoOm\subseteq\MqtOm&\mathrm{for}\hspace{2mm}p^*\leq q_2\leq q_1\leq\infty.
\end{array}\right.\nonumber
\end{equation}
When $p=N$, we have
\begin{equation}
\MinOm\subseteq\MOm=\Wop\hspace{2mm}\mathrm{for}\hspace{2mm}1\leq q<\infty.\nonumber
\end{equation}
When $p\in\left(N,\infty\right]$, we have
\begin{equation}
\MOm=\Wop\hspace{2mm}\mathrm{for}\hspace{2mm}1\leq q\leq\infty.\nonumber
\end{equation}

\begin{lem}\label{L2.1}
Assume $1\leq q\leq\infty$.
When $1\leq p<N$, then the embedding $\iota:\MOm\to\LsOm$ is continuous if $1\leq s\leq\max\left\{p^*,q\right\}$ and compact if $1\leq s<\max\left\{p^*,q\right\}$.
When $p=N$, then the embedding $\iota:\MOm\hookrightarrow\LsOm$ is compact if $1\leq s<\infty$.
When $N<p\leq\infty$, then the embedding $\iota:\MOm\to\LsOm$ is continuous if $1\leq s\leq\infty$ and compact if $1\leq s<\infty$.
\end{lem}

Next, consider $\Om=\RN$.
When $p\in\left[1,N\right)$, we denote by $\DpR$ the space of functions $u$ with $u\in\LpcR$ and $\n{\gradu}\in\LpR$; through \textsf{Gagliardo-Nirenberg-Sobolev inequality}, there exists a sharp constant $C_1>0$, depending on $N,p$, such that
\begin{equation}\label{eq2.2}
\nm{u}_{p^*,\hspace{0.2mm}\RN}\leq C_1\nm{\n{\gradu}}_{p,\hspace{0.2mm}\RN},\hspace{6mm}\forall\hspace{2mm}u\in\DpR.
\end{equation}
This leads to $\DpR=\McR$ and we in addition have
\begin{equation}
\MqoR\subseteq\MqtR\hspace{2mm}\mathrm{for}\hspace{2mm}
\begin{array}{ll}
\mathrm{either}&1\leq q_1\leq q_2\leq p^*,\medskip\\
\mathrm{or}&p^*\leq q_2\leq q_1<\infty.
\end{array}\nonumber
\end{equation}
When $p=N$, we have
\begin{equation}
\MqoR\subseteq\MqtR\hspace{2mm}\mathrm{for}\hspace{2mm}1\leq q_1\leq q_2<\infty.\nonumber
\end{equation}
When $p\in\left(N,\infty\right]$, we have
\begin{equation}
\MqoR\subseteq\MqtR\subseteq\MinR\hspace{2mm}\mathrm{for}\hspace{2mm}1\leq q_1\leq q_2<\infty.\nonumber
\end{equation}

\begin{lem}\label{L2.2}
Assume that $1\leq q<\infty$.
When $1\leq p<N$, then the embedding $\iota:\MRN\to\LsR$ is continuous if $\min\left\{p^*,q\right\}\leq s\leq\max\left\{p^*,q\right\}$.
When $p=N$, then the embedding $\iota:\MRN\to\LsR$ is continuous if $q\leq s<\infty$.
When $N<p\leq\infty$, then the embedding $\iota:\MRN\to\LsR$ is continuous if $q\leq s\leq\infty$.
\end{lem}

The class $\CCON$ of compactly supported, continuously differentiable functions provides a dense subset of $\MRN$ in all cases except when $p,q=\infty$.
Moreover, one has the following profound \textsf{Caffarelli-Kohn-Nirenberg inequality} from Rabier \cite[corollary 2.1]{Rab}
\begin{equation}\label{eq2.3}
\nm{u}_{s,\hspace{0.2mm}\RN}\leq C_2\nm{u}^{1-\theta}_{q,\hspace{0.2mm}\RN}\nm{\n{\gradu}}^{\theta}_{p,\hspace{0.2mm}\RN},\hspace{6mm}\forall\hspace{2mm}u\in\MRN.
\end{equation}
Here, $1\leq q\left(\neq p^*\right)<\infty$ and $s$ lies in between $p^*$ and $q$ if $1\leq p<N$ whereas $1\leq q\leq s<\infty$ if $N\leq p<\infty$, $\theta=\frac{Nps-Npq}{Nps+pqs-Nqs}\in\left[0,1\right)$, and $C_2>0$ is a constant depending on $N,p,q,s$.

All the preceding results can be found with details in \cite{Rab} and \cite{Ha1,Ha2,Ha3}.
Note some special cases of \eqref{eq2.3} were proved independently by Brasco and Ruffini \cite[proposition 2.6]{BR}.

Below, we discuss some compact embedding results for $\MVR\hookrightarrow\LKR$.

\begin{prop}\label{P2.3}
Assume $1\leq p<N$, $1<q<\infty$, $1\leq r<\min\left\{p^*,q\right\}$, and $K(x),V(x)>0$ satisfy $K(x)\in\LalocR$ for some $\alpha\in\left(\frac{\max\left\{p^*,q\right\}}{\max\left\{p^*,q\right\}-r},\infty\right]$, $\inf\limits_{\mathbf{D}}V(x)\geq V_{\mathbf{D}}>0$ for each compact subset $\mathbf{D}\Subset\RN$ whereas $K^{\frac{p^*+\beta\left(p^*-q\right)}{p^*-r}}(x)V^{-\beta}(x)\in\LlR$ for some $\beta\in\left[0,\frac{r}{q-r}\right]$.
Then, the embedding $\MVR\hookrightarrow\LKR$ is compact.
\end{prop}

This result unifies and extends proposition 2.3 and theorem 4.6 in \cite{Ha3}.

\begin{proof}
Recall that $\MVR$ is a subspace of $\DpR$ by definition.
Write $\mathfrak{x}=\frac{\beta\left(p^*-r\right)}{p^*+\beta\left(p^*-q\right)}$, $\mathfrak{y}=\frac{r+\beta\left(r-q\right)}{p^*+\beta\left(p^*-q\right)}$ and $\mathfrak{z}=\frac{p^*-r}{p^*+\beta\left(p^*-q\right)}$, and notice that $\mathfrak{x}+\mathfrak{y}+\mathfrak{z}=1$.
Now, set $r_1=\mathfrak{x}^{-1}$, $r_2=\mathfrak{y}^{-1}$ and $r_3=\mathfrak{z}^{-1}$ to observe
\begin{equation}\label{eq2.4}
\begin{split}
&\,\IR\n{u}^rK\dx=\IR\left\{\n{u}^qV\right\}^{\mathfrak{x}}\n{u}^{r-q\mathfrak{x}}\left\{KV^{-\mathfrak{x}}\right\}\dx\\
\leq&\left(\IR\n{u}^qV\dx\right)^{\frac{1}{r_1}}\left(\IR\n{u}^{p^*}\dx\right)^{\frac{1}{r_2}}\left(\IR K^{r_3}V^{-\beta}\dx\right)^{\frac{1}{r_3}}\\
\leq&\,\,\widetilde{C}_1\left(\IR\n{u}^qV\dx\right)^{\frac{1}{r_1}}\left(\IR\n{\gradu}^p\dx\right)^{\frac{p^*}{pr_2}}\left(\IR K^{r_3}V^{-\beta}\dx\right)^{\frac{1}{r_3}}
\end{split}
\end{equation}
via \textsf{H\"{o}lder}'s \textsf{inequality} and \eqref{eq2.2} with $p^*=r_2\left(r-q\mathfrak{x}\right)$, provided $\mathfrak{x},\mathfrak{y},\mathfrak{z}\in\left(0,1\right)$.
To have $\mathfrak{x},\mathfrak{y},\mathfrak{z}\in\left(0,1\right)$, one can simply repeat the discussions in \cite[proposition 2.3]{Ha3} to show $\beta\in\left(0,\frac{r}{q-r}\right)$.
We certainly can take $\mathfrak{x},\mathfrak{y}=0$ and consequently have $\beta=0,\frac{r}{q-r}$.
Notice $\beta=0$ corresponds to the case where $V(x)\equiv0$ and $\MVR=\DpR$, which is \cite[theorem 4.6]{Ha3}.

Next, one has $\frac{q}{r_1}+\frac{p^*}{r_2}=q\mathfrak{x}+p^*\mathfrak{y}=r$ and therefore the embedding $\MVR\to\LKR$ is continuous.
Now, let $\left\{u_k:k\geq1\right\}$ be a sequence of functions in $\MVR$, with $u_k\rightharpoonup0$ as $k\to\infty$ and $\nm{u_k}_{\MVR}$ uniformly bounded.
It follows that
\begin{equation}\label{eq2.5}
\IR\n{u_k}^rK\dx=\IBR\n{u_k}^rK\dx+\IBRc\n{u_k}^rK\dx.
\end{equation}
Here, and hereafter, $\BR$ denotes the ball of radius $R$ in $\RN$ that is centered at the origin and $\BRc:=\RN\setminus\BR$.
For the integral over $\BRc$, we apply \eqref{eq2.4} to derive, as $R\to\infty$,
\begin{equation}
\IBRc\n{u_k}^rK\dx\leq\widetilde{C}_1\nm{K^{r_3}V^{-\beta}}^{\frac{1}{r_3}}_{1,\hspace{0.2mm}\BRc}\nm{u_k}^r_{\MVR}\to0.\nonumber
\end{equation}
For the integral over $\BR$, our (local) hypotheses lead to the compact embedding
\begin{equation}\label{eq2.6}
\MVBR\to\MBR\hookrightarrow L^{\frac{\alpha r}{\alpha-1}}\left(\BR\right)\to\LKBR
\end{equation}
by virtue of lemma \ref{L2.1} since $\frac{\alpha r}{\alpha-1}\in\left[1,\max\left\{p^*,q\right\}\right)$; as a result, $\nm{u_k}_{\LKBR}\to0$ when $k\to\infty$ for a subsequence relabeled with the same index $k$.
So, $u_k\to0$ in $\LKR$.
\end{proof}

\begin{prop}\label{P2.4}
Assume $1\leq p<N$, $p^*\leq r<q<\infty$, and $K(x),V(x)>0$ satisfy $K(x)\in\LalocR$ for some $\alpha\in\left(\frac{q}{q-r},\infty\right]$, $\inf\limits_{\mathbf{D}}V(x)\geq V_{\mathbf{D}}>0$ for all compact subsets $\mathbf{D}\Subset\RN$ while $K^{\frac{q}{q-r}}(x)V^{-\frac{r}{q-r}}(x)\in\LlR$.
Then, the embedding $\MVR\hookrightarrow\LKR$ is compact.
\end{prop}

\begin{proof}
One observes from \textsf{H\"{o}lder}'s \textsf{inequality} that
\begin{equation}\label{eq2.7}
\IR\n{u}^rK\dx\leq\left(\IR\n{u}^qV\dx\right)^{\frac{r}{q}}\left(\IR K^{\frac{q}{q-r}}V^{-\frac{r}{q-r}}\dx\right)^{\frac{q-r}{q}}.
\end{equation}
The continuity of the embedding $\MVR\to\LKR$ follows readily.
Let $\left\{u_k:k\geq1\right\}$ be a sequence of functions in $\MVR$, with $u_k\rightharpoonup0$ when $k\to\infty$ and $\nm{u_k}_{\MVR}$ uniformly bounded.
One has \eqref{eq2.5} so that for the integral over $\BRc$, it yields, as $R\to\infty$,
\begin{equation}
\IBRc\n{u_k}^rK\dx\leq\nm{K^{\frac{q}{q-r}}V^{-\frac{r}{q-r}}}^{\frac{q-r}{q}}_{1,\hspace{0.2mm}\BRc}\nm{u_k}^r_{\MVR}\to0;\nonumber
\end{equation}
for the integral over $\BR$, noticing $1\leq\frac{\alpha r}{\alpha-1}<q$, our hypotheses again imply \eqref{eq2.6} by virtue of lemma \ref{L2.1}.
As a consequence, one analogously obtains $u_k\to0$ in $\LKR$.
\end{proof}

The following result provides a different version of theorem 4.3 in \cite{Ha3}.

\begin{thm}\label{T2.5}
Suppose $1\leq p<N$, $p^*\leq r<q<\infty$, and $K(x),V(x)>0$ satisfy $K(x)\in\LalocR$ for some $\alpha\in\left(\frac{q}{q-r},\infty\right]$, $\inf\limits_{\mathbf{D}}V(x)\geq V_{\mathbf{D}}>0$ for all compact subsets $\mathbf{D}\Subset\RN$ while $K(x)V^{-\frac{r-p^*}{q-p^*}}(x)\to0$ uniformly.
Then, the embedding $\MVR\hookrightarrow\LKR$ is compact.
\end{thm}

\begin{proof}
Assume $\bre{u_k:k\geq1}$ is a sequence of functions in $\MVR$, with $u_k\rightharpoonup0$ as $k\to\infty$ and $\nm{u_k}_{\MVR}$ uniformly bounded.
One has \eqref{eq2.5} and for the integral over $\BRc$,
\begin{equation}
\begin{split}
&\,\IBRc\n{u_k}^rK\dx=\IBRc\left\{\n{u_k}^qV\right\}^{\frac{r-p^*}{q-p^*}}\n{u_k}^{r-q\frac{r-p^*}{q-p^*}}\left\{KV^{-\frac{r-p^*}{q-p^*}}\right\}\dx\\
\leq&\nm{KV^{-\frac{r-p^*}{q-p^*}}}_{\infty,\hspace{0.2mm}\BRc}\left(\IBRc\n{u_k}^qV\dx\right)^{\frac{r-p^*}{q-p^*}}\left(\IBRc\n{u_k}^{p^*}\dx\right)^{\frac{q-r}{q-p^*}}\\
\leq&\,\,\widehat{C}_1\nm{KV^{-\frac{r-p^*}{q-p^*}}}_{\infty,\hspace{0.2mm}\BRc}\left(\IR\n{u_k}^qV\dx\right)^{\frac{r-p^*}{q-p^*}}
\left(\IR\n{\gradu_k}^p\dx\right)^{\frac{p^*\left(q-r\right)}{p\left(q-p^*\right)}}\\
\leq&\,\,\widehat{C}_1\nm{KV^{-\frac{r-p^*}{q-p^*}}}_{\infty,\hspace{0.2mm}\BRc}\nm{u_k}^r_{\MVR}\nonumber
\end{split}
\end{equation}
follows in view of \textsf{H\"{o}lder}'s \textsf{inequality} and \eqref{eq2.2} that goes to zero when $R\to\infty$; the analysis on the integral over $\BR$ is exactly the same as done before so that $u_k\to0$ in $\LKR$.

Note the embedding $\MVR\to\LKR$ is continuous provided $K(x)\in\LalocR$ for some $\alpha\in\left[\frac{q}{q-r},\infty\right]$ and $K(x)V^{-\frac{r-p^*}{q-p^*}}(x)$ is eventually bounded as $\n{x}\to\infty$.
\end{proof}

\begin{thm}\label{T2.6}
Suppose $N\leq p<\infty$, $1\leq r<q<\infty$, and $K(x),V(x)>0$ satisfy $K(x)\in\LalocR$ for some $\alpha\in\left(1,\infty\right]$, $\inf\limits_{\RN}V(x)\geq V_0>0$ while $K^{\frac{s+\beta\left(s-q\right)}{s-r}}(x)V^{-\beta}(x)\in\LlR$ for some $\beta\in\left[0,\frac{r}{q-r}\right]$ and $s\in\left[q,\infty\right)$.
Then, the embedding $\MVR\hookrightarrow\LKR$ is compact.
\end{thm}

\begin{proof}
Recall that $\MVR$ is a subspace of $\MRN$ by definition.
Write $\hat{\mathfrak{x}}=\frac{\beta\left(s-r\right)}{s+\beta\left(s-q\right)}$, $\hat{\mathfrak{y}}=\frac{r+\beta\left(r-q\right)}{s+\beta\left(s-q\right)}$ and $\hat{\mathfrak{z}}=\frac{s-r}{s+\beta\left(s-q\right)}$ for some arbitrarily chosen $s\in\left[q,\infty\right)$, and notice $\hat{\mathfrak{x}}+\hat{\mathfrak{y}}+\hat{\mathfrak{z}}=1$.
Now, set $\hat{r}_1=\hat{\mathfrak{x}}^{-1}$, $\hat{r}_2=\hat{\mathfrak{y}}^{-1}$ and $\hat{r}_3=\hat{\mathfrak{z}}^{-1}$ for $\theta=\frac{Nps-Npq}{Nps+pqs-Nqs}\in\left[0,1\right)$ to derive
\begin{equation}\label{eq2.8}
\begin{split}
&\,\IR\n{u}^rK\dx=\IR\left\{\n{u}^qV\right\}^{\hat{\mathfrak{x}}}\n{u}^{r-q\hat{\mathfrak{x}}}\left\{KV^{-\hat{\mathfrak{x}}}\right\}\dx\\
\leq&\left(\IR\n{u}^qV\dx\right)^{\frac{1}{\hat{r}_1}}\left(\IR\n{u}^s\dx\right)^{\frac{1}{\hat{r}_2}}\left(\IR K^{\hat{r}_3}V^{-\beta}\dx\right)^{\frac{1}{\hat{r}_3}}\\
\leq&\,\,C(s)V^{-\frac{s\left(1-\theta\right)}{q\hat{r}_2}}_0\left(\IR\n{u}^qV\dx\right)^{\frac{1}{\hat{r}_1}+\frac{s\left(1-\theta\right)}{q\hat{r}_2}}
\left(\IR\n{\gradu}^p\dx\right)^{\frac{s\theta}{p\hat{r}_2}}\left(\IR K^{\hat{r}_3}V^{-\beta}\dx\right)^{\frac{1}{\hat{r}_3}}
\end{split}
\end{equation}
by \textsf{H\"{o}lder}'s \textsf{inequality} and \eqref{eq2.3} with $s=\hat{r}_2\left(r-q\hat{\mathfrak{x}}\right)$ when $\hat{\mathfrak{x}},\hat{\mathfrak{y}},\hat{\mathfrak{z}}\in\left(0,1\right)$.
To have $\hat{\mathfrak{x}},\hat{\mathfrak{y}},\hat{\mathfrak{z}}\in\left(0,1\right)$, one deduces $\beta\in\left(0,\frac{r}{q-r}\right)$.
We surely can take $\hat{\mathfrak{x}},\hat{\mathfrak{y}}=0$ and have $\beta=0,\frac{r}{q-r}$.
As
\begin{equation}\label{eq2.9}
q\left(\frac{1}{\hat{r}_1}+\frac{s\left(1-\theta\right)}{q\hat{r}_2}\right)+p\left(\frac{s\theta}{p\hat{r}_2}\right)
=\frac{q}{\hat{r}_1}+\frac{s}{\hat{r}_2}=q\hat{\mathfrak{x}}+s\hat{\mathfrak{y}}=r,
\end{equation}
one realizes that the embedding $\MVR\to\LKR$ is continuous.

Next, let $\left\{u_k:k\geq1\right\}$ be a sequence of functions in $\MVR$ such that $u_k\rightharpoonup0$ as $k\to\infty$ and $\nm{u_k}_{\MVR}$ is uniformly bounded.
Using \eqref{eq2.5}, for the integral over $\BRc$, we apply \eqref{eq2.8} and \eqref{eq2.9} to derive, when $R\to\infty$,
\begin{equation}
\IBRc\n{u_k}^rK\dx\leq C(s)V^{-\frac{s\left(1-\theta\right)}{q\hat{r}_2}}_0\nm{K^{\hat{r}_3}V^{-\beta}}^{\frac{1}{\hat{r}_3}}_{1,\hspace{0.2mm}\BRc}\nm{u_k}^r_{\MVR}\to0;\nonumber
\end{equation}
for the integral over $\BR$, noticing $1\leq\frac{\alpha r}{\alpha-1}<\infty$, our hypotheses again lead to \eqref{eq2.6} in view of lemma \ref{L2.1}.
As a consequence, one analogously observes $u_k\to0$ in $\LKR$.
\end{proof}

It is noteworthy that our preceding results in particular provide some complements to those nice results by Chiappinelli \cite{Chi} in the so-called \textsl{lower triangle} situation.
Moreover, proposition \ref{P2.3} here is related to (and seems providing a correct proof for) Schneider \cite[theorem 2.3]{Sc}, but the author wasn't aware of that paper when this paper was initially written.

% +++ below is section 3 %%%%%%%%%%%%%%%%%%%%%%%%%%%%%%%%%%%%%%%%%%%%%%%%%%%%%%%%%%%%%%%%%%%%%%%%%%%%%%%%%%%%%%%%%%%%%%%%%%%%%%%%%%%%%%%%%%%%%%%%%%%
% &&& %%%%%%%%%%%%%%%%%%%%%%%%%%%%%%%%%%%%%%%%%%%%%%%%%%%%%%%%%%%%%%%%%%%%%%%%%%%%%%%%%%%%%%%%%%%%%%%%%%%%%%%%%%%%%%%%%%%%%%%%%%%%%%%%%%%%%%%%%%%%%%
\section{Proof of Theorem \ref{T1.1}}\label{PT1.1} % use lowercase except for proper names
\noindent In this section, we seek nontrivial positive solutions to \eqref{eq1.2} in $\MVR$ via identifying the critical points of the associated energy functional $\Jc:\MVR\to\R$, defined by
\begin{equation}\label{eq3.1}
\Jc(u)=\frac{1}{p}\IR\n{\gradu}^p\dx+\frac{1}{q}\IR\n{u}^qV\dx-\frac{\lambda}{r}\IR\n{u}^rK\dx.
\end{equation}

First, we make an elementary observation of all solutions to problem $\eqref{eq1.2}$.

\begin{lem}\label{L3.1}
Under the {\bf Standing Assumptions (i)-(iii)}, each solution $\ul$ to equation $\eqref{eq1.2}$ in $\MVR$ satisfies
\begin{equation}\label{eq3.2}
\IR\n{\gradul}^p\dx+\IR\n{\ul}^qV\dx\leq\lambda^{\gamma}C_{KV}.
\end{equation}
Here, $\gamma>0$ and $C_{KV}>0$ are absolute constants that are independent of $\lambda,u$.
\end{lem}

\begin{proof}
First, it's easily seen that each solution $\ul$ to equation $\eqref{eq1.2}$ satisfies
\begin{equation}\label{eq3.3}
\IR\n{\gradul}^p\dx+\IR\n{\ul}^qV\dx=\lambda\IR\n{\ul}^rK\dx.
\end{equation}

For $1<p<N$ and $p<r<\min\left\{p^*,q\right\}$, denote $r_4=\frac{p^*+\beta\left(p^*-q\right)}{\beta\left(p^*-r\right)}=r_1$, $r_5=\frac{p\left\{p^*+\beta\left(p^*-q\right)\right\}}{p^*\left\{r+\beta\left(r-q\right)\right\}}=\frac{p}{p^*}r_2<r_2$ and $r_6=\frac{r_4r_5}{r_4r_5-r_4-r_5}>r_3$ in \eqref{eq2.4} of proposition \ref{P2.3} to observe
\begin{equation}\label{eq3.4}
\begin{split}
&\,\lambda\IR\n{u}^rK\dx\leq\left(\IR\n{u}^qV\dx\right)^{\frac{1}{r_4}}\left(\IR\n{\gradu}^p\dx\right)^{\frac{1}{r_5}}\times\\
&\hspace{28.4mm}\left\{\lambda^{r_6}\widetilde{C}^{r_6}_1\left(\IR K^{r_3}V^{-\beta}\dx\right)^{\frac{r_6}{r_3}}\right\}^{\frac{1}{r_6}}\\
\leq&\,\,\frac{1}{r_4}\IR\n{u}^qV\dx+\frac{1}{r_5}\IR\n{\gradu}^p\dx+\lambda^{r_6}C_1(K,V).
\end{split}
\end{equation}
Here, we applied \textsf{Young}'s \textsf{inequality} with $C_1(K,V)>0$ a constant independent of $\lambda,u$.
\eqref{eq3.2} is verified via \eqref{eq3.3} if $\frac{1}{r_4}+\frac{1}{r_5}<1$, which is true provided $\frac{p^*\left(r-p\right)}{\left(p^*-p\right)\left(q-r\right)}<\beta\leq\frac{r}{q-r}$.

For $1<p<N$ but $p^*\leq r<q<\infty$, \eqref{eq2.7} of proposition \ref{P2.4} immediately yields
\begin{equation}\label{eq3.5}
\lambda\IR\n{u}^rK\dx\leq\frac{r}{q}\IR\n{u}^qV\dx+\lambda^{\frac{q}{q-r}}\widetilde{C}_1(K,V),
\end{equation}
with $\widetilde{C}_1(K,V)>0$ a constant independent of $\lambda,u$.

Finally, for $N\leq p<\infty$ and $p<r<q<\infty$, \eqref{eq2.8} of theorem \ref{T2.6} leads to
\begin{equation}\label{eq3.6}
\begin{split}
&\,\lambda\IR\n{u}^rK\dx\leq\left(\IR\n{u}^qV\dx\right)^{\frac{1}{\hat{r}_4}}\left(\IR\n{\gradu}^p\dx\right)^{\frac{1}{\hat{r}_5}}\times\\
&\hspace{28.4mm}\left\{\left(\lambda\hspace{0.2mm}C(s)V^{-\frac{s\left(1-\theta\right)}{q\hat{r}_2}}_0\right)^{\hat{r}_6}
\left(\IR K^{\hat{r}_3}V^{-\beta}\dx\right)^{\frac{\hat{r}_6}{\hat{r}_3}}\right\}^{\frac{1}{\hat{r}_6}}\\
\leq&\,\,\frac{1}{\hat{r}_4}\IR\n{u}^qV\dx+\frac{1}{\hat{r}_5}\IR\n{\gradu}^p\dx+\lambda^{\hat{r}_6}\widehat{C}_1(K,V)
\end{split}
\end{equation}
through \textsf{Young}'s \textsf{inequality} with $\widehat{C}_1(K,V)>0$ a constant independent of $\lambda,u$ for $\hat{r}_4=\frac{1}{\frac{1}{\hat{r}_1}+\frac{s\left(1-\theta\right)}{q\hat{r}_2}}$, $\hat{r}_5=\frac{1}{\frac{s\theta}{p\hat{r}_2}}$ and $\hat{r}_6=\frac{\hat{r}_4\hat{r}_5}{\hat{r}_4\hat{r}_5-\hat{r}_4-\hat{r}_5}$.
To have \textsf{Young}'s \textsf{inequality} applicable, we simply require $\hat{r}_6>0$, or equivalently, $\frac{1}{\hat{r}_4}+\frac{1}{\hat{r}_5}<1$.
Routine calculations lead to
\begin{equation}
\begin{split}
\frac{1}{\hat{r}_4}+\frac{1}{\hat{r}_5}=&\,\,\frac{1}{\hat{r}_1}+\frac{s\left(1-\theta\right)}{q\hat{r}_2}+\frac{s\theta}{p\hat{r}_2}\\
=&\,\,\frac{\beta Nps+\beta Nq^2+\beta prs-\beta Npq-\beta Nqs-\beta pqr+Npr+prs-Nqr}{\beta Nps+\beta Nq^2+\beta pqs-\beta Npq-\beta Nqs-\beta pq^2+Nps+pqs-Nqs},\nonumber
\end{split}
\end{equation}
so that $\frac{1}{\hat{r}_4}+\frac{1}{\hat{r}_5}<1$ if and only if
\begin{equation}
\beta>\mathfrak{f}(s):=\frac{Npr+Nqs+prs-Nps-Nqr-pqs}{p\left(q-r\right)\left(s-q\right)}.\nonumber
\end{equation}
It is readily seen that $\mathfrak{f}(s)<\frac{r}{q-r}$ provided $s>r$, since $Np+pq-Nq\geq Np>0$ in view of the assumption $p\geq N$.
Furthermore, it is interesting to derive that
\begin{equation}
\begin{split}
\mathfrak{f}'(s)=&\,\,\frac{Np^2q^2+Np^2r^2-2Np^2qr+p^2q^3+p^2qr^2-2p^2q^2r-Npq^3-Npqr^2+2Npq^2r}{\bre{p\left(q-r\right)\left(s-q\right)}^2}\\
=&\,\,\frac{p\left(Np+pq-Nq\right)\left(q^2+r^2-2qr\right)}{\bre{p\left(q-r\right)\left(s-q\right)}^2}\geq\frac{N}{\left(s-q\right)^2}>0\nonumber
\end{split}
\end{equation}
and $\lim\limits_{s\to q^+}\mathfrak{f}(s)=-\frac{\left(q-r\right)\left(Np+pq-Nq\right)}{\lim\limits_{s\to q^+}\bre{p\left(q-r\right)\left(s-q\right)}}=-\infty$.
Hence, $\mathfrak{f}(s)$ is an increasing function of $s\in\left(q,\infty\right)$ with supremum $\lim\limits_{s\to\infty^-}\mathfrak{f}(s)=\frac{Nq+pr-Np-pq}{p\left(q-r\right)}<\frac{r}{q-r}$.
To have the largest lower bound regarding $\left(\mathfrak{f}(s),\frac{r}{q-r}\right]$, we may take $s\to\infty$ to derive $\max\left\{0,\frac{Nq+pr-Np-pq}{p\left(q-r\right)}\right\}<\beta\leq\frac{r}{q-r}$.

Note $r_4,r_5$ in \eqref{eq3.4} depend only on $\beta,N,p,q,r$.
So, we analyze $\hat{r}_4,\hat{r}_5$ in \eqref{eq3.6} (as functions of $s$) to remove their dependence on $s$.
Suppose $\beta\in\left(\max\left\{0,\frac{Nq+pr-Np-pq}{p\left(q-r\right)}\right\},\frac{r}{q-r}\right]$ subsequently and have $\hat{r}_4,\hat{r}_5>1$ uniformly for $s\in\left[q,\infty\right]$.
Recall
\begin{equation}
\begin{split}
\mathfrak{g}(s):=&\,\,\frac{1}{\hat{r}_4}=\frac{1}{\hat{r}_1}+\frac{s\left(1-\theta\right)}{q\hat{r}_2}\\
=&\,\,\frac{\beta Nps+\beta Nqr+\beta prs-\beta Npq-\beta Nrs-\beta pqr+Npr+prs-Nrs}{\beta Nps+\beta Nq^2+\beta pqs-\beta Npq-\beta Nqs-\beta pq^2+Nps+pqs-Nqs}\nonumber
\end{split}
\end{equation}
and
\begin{equation}
\begin{split}
\mathfrak{h}(s):=&\,\,\frac{1}{\hat{r}_5}=\frac{s\theta}{p\hat{r}_2}\\
=&\,\,\frac{\beta Nq^2+\beta Nrs-\beta Nqr-\beta Nqs+Nrs-Nqr}{\beta Nps+\beta Nq^2+\beta pqs-\beta Npq-\beta Nqs-\beta pq^2+Nps+pqs-Nqs}.\nonumber
\end{split}
\end{equation}
It is straightforward (but somewhat tedious) to obtain that
\begin{equation}
\begin{split}
\mathfrak{g}'(s)=&\,\,\frac{\beta N^2p^2q+\beta N^2pqr+\beta Np^2q^2+N^2pqr}{\left\{\beta Nps+\beta Nq^2+\beta pqs-\beta Npq-\beta Nqs-\beta pq^2+Nps+pqs-Nqs\right\}^2}\\
&\,-\frac{\beta N^2p^2r+\beta N^2pq^2+\beta Np^2qr+N^2p^2r+Np^2qr}{\left\{\beta Nps+\beta Nq^2+\beta pqs-\beta Npq-\beta Nqs-\beta pq^2+Nps+pqs-Nqs\right\}^2}\\
=&\,-\frac{Np\left(Np+pq-Nq\right)\left\{r+\beta\left(r-q\right)\right\}}{\left\{s+\beta\left(s-q\right)\right\}^2\left(Np+pq-Nq\right)^2}<0\nonumber
\end{split}
\end{equation}
and
\begin{equation}
\begin{split}
\mathfrak{h}'(s)=&\,\,\frac{\beta N^2pqr+\beta N^2q^3+\beta Npq^2r+N^2pqr+Npq^2r}{\left\{\beta Nps+\beta Nq^2+\beta pqs-\beta Npq-\beta Nqs-\beta pq^2+Nps+pqs-Nqs\right\}^2}\\
&\,-\frac{\beta N^2pq^2+\beta N^2q^2r+\beta Npq^3+N^2q^2r}{\left\{\beta Nps+\beta Nq^2+\beta pqs-\beta Npq-\beta Nqs-\beta pq^2+Nps+pqs-Nqs\right\}^2}\\
=&\,\,\frac{Nq\left(Np+pq-Nq\right)\left\{r+\beta\left(r-q\right)\right\}}{\left\{s+\beta\left(s-q\right)\right\}^2\left(Np+pq-Nq\right)^2}>0.\nonumber
\end{split}
\end{equation}
Therefore, \eqref{eq3.6} is transformed to
\begin{equation}\label{eq3.7}
\lambda\IR\n{u}^rK\dx\leq\frac{r}{q}\IR\n{u}^qV\dx+\frac{1}{\delta}\IR\n{\gradu}^p\dx+\lambda^{\hat{r}_6}\widehat{C}_1(K,V),
\end{equation}
since $\frac{1}{\hat{r}_4}\leq\lim\limits_{s\to q^+}\mathfrak{g}(s)=\frac{r}{q}<1$ and $\frac{1}{\hat{r}_5}\leq\frac{1}{\delta}:=\lim\limits_{s\to\infty^-}\mathfrak{h}(s)=\frac{N\left\{r+\beta\left(r-q\right)\right\}}{\left(1+\beta\right)\left(Np+pq-Nq\right)}<1$.
\end{proof}

\begin{lem}\label{L3.2}
Under the {\bf Standing Assumptions (i)-(iii)}, the functional $\Jc$ is of class $C^1$ and is coercive in $\MVR$ so that each sequence $\bre{u_k:k\geq1}$ of functions in $\MVR$ with $\Jc(u_k)$ bounded admits of a weakly convergent subsequence in $\MVR$.
Furthermore, $\Jc$ is sequentially weakly lower semicontinuous in $\MVR$; that is, when $u_k\rightharpoonup u$ in $\MVR$, then for a subsequence relabeled using the same notation, one has
\begin{equation}\label{eq3.8}
\Jc(u)\leq\varliminf_{k\to\infty}\Jc(u_k).
\end{equation}
\end{lem}

\begin{proof}
The proof of showing $\Jc$ is $C^1$ is standard.
The assertion regarding the boundedness of $\Jc(u_k)$ leading to the existence of a weakly convergent subsequence in $\MVR$ follows via the coercivity of $\Jc$ and the reflexivity of $\MVR$; for the latter, see proposition a.11 of \cite{AP} with $E$ now being the homogeneous gradient $L^p$ space using the notation there.

Next, we show $\Jc$ is coercive.
For $1<p<N$ and $p<r<\min\left\{p^*,q\right\}$, one has
\begin{equation}
\begin{split}
&\,\frac{\lambda}{r}\IR\n{u}^rK\dx\leq\left(\frac{r_4}{2q}\IR\n{u}^qV\dx\right)^{\frac{1}{r_4}}\left(\frac{r_5}{2p}\IR\n{\gradu}^p\dx\right)^{\frac{1}{r_5}}\times\\
&\hspace{29.1mm}\left\{\lambda^{r_6}\left(\frac{\widetilde{C}_1}{r}\right)^{r_6}\left(\frac{2q}{r_4}\right)^{\frac{r_6}{r_4}}\left(\frac{2p}{r_5}\right)^{\frac{r_6}{r_5}}
\left(\IR K^{r_3}V^{-\beta}\dx\right)^{\frac{r_6}{r_3}}\right\}^{\frac{1}{r_6}}\\
\leq&\,\,\frac{1}{2q}\IR\n{u}^qV\dx+\frac{1}{2p}\IR\n{\gradu}^p\dx+\lambda^{r_6}C_2(K,V).\nonumber
\end{split}
\end{equation}
For $1<p<N$ but $p^*\leq r<q<\infty$, one has
\begin{equation}
\frac{\lambda}{r}\IR\n{u}^rK\dx\leq\frac{1}{2q}\IR\n{u}^qV\dx+\lambda^{\frac{q}{q-r}}\widetilde{C}_2(K,V).\nonumber
\end{equation}
Finally, for $N\leq p<\infty$ and $p<r<q<\infty$, one has
\begin{equation}
\frac{\lambda}{r}\IR\n{u}^rK\dx\leq\frac{1}{2q}\IR\n{u}^qV\dx+\frac{1}{2p}\IR\n{\gradu}^p\dx+\lambda^{\hat{r}_6}\widehat{C}_2(K,V).\nonumber
\end{equation}
Here, we employed the same ideas and notations as used in lemma \ref{L3.1}, so that
\begin{equation}
\Jc(u)\geq\frac{1}{2p}\IR\n{\gradu}^p\dx+\frac{1}{2q}\IR\n{u}^qV\dx-\lambda^{\gamma}\widetilde{C}_{KV}\nonumber
\end{equation}
with (the same) $\gamma>0$ and $\widetilde{C}_{KV}>0$ some absolute constants independent of $\lambda,u$.

Notice we have just proved the coercivity of $\Jc$.
So, each sequence $\bre{u_k:k\geq1}$ of functions in $\MVR$ with bounded $\Jc(u_k)$ admits of a weakly convergent subsequence, written again as $\bre{u_k:k\geq1}$ with $u_k\rightharpoonup u\in\MVR$.
The lower semicontinuity of norms yields
\begin{equation}
\IR\n{\gradu}^p\dx\leq\varliminf_{k\to\infty}\IR\n{\gradu_k}^p\dx\hspace{2mm}\mathrm{and}\hspace{2mm}\IR\n{u}^qV\dx\leq\varliminf_{k\to\infty}\IR\n{u_k}^qV\dx,\nonumber
\end{equation}
while Lieb and Loss \cite[theorem 1.9]{LL} (together with the fact that $u_k\sqrt[r]{K}\to u\sqrt[r]{K}$ \textsl{a.e.} on $\RN$ by lemma \ref{L2.1} for yet another subsequence, still denoted by $\bre{u_k:k\geq1}$) says
\begin{equation}\label{eq3.9}
\lim_{k\to\infty}\IR\n{u_k}^rK\dx=\IR\n{u}^rK\dx.
\end{equation}
Therefore, one proves the sequentially weak lower semicontinuity \eqref{eq3.8} of $\Jc$.
\end{proof}

Define
\begin{equation}\label{eq3.10}
\tilde{\lambda}=\inf_{\substack{u\in\MVR \\ \nm{u}_{\LKR}=1}}\bre{\frac{r}{p}\IR\n{\gradu}^p\dx+\frac{r}{q}\IR\n{u}^qV\dx}.
\end{equation}
Then, one sees $\tilde{\lambda}>0$.
Actually, if not, then there is a sequence $\bre{u_l:l\geq1}$ in $\MVR$ with $\nm{u_l}_{\LKR}=1$ but $\frac{r}{p}\IR\n{\gradu_l}^p\dx+\frac{r}{q}\IR\n{u_l}^qV\dx\to0$; this would yield $\nm{u_l}_{\MVR}\to0$ that contradicts the compact embedding $\MVR\hookrightarrow\LKR$ or \eqref{eq3.9}.

Define $\lambda^*$ to be the supermum of $\lambda$ such that equation \eqref{eq1.2} only has the trivial solution for each $\mu<\lambda$, and $\lambda^{**}$ to be the infimum of $\lambda$ such that equation \eqref{eq1.2} has at least one nontrivial positive solution at $\lambda$.
Then, we have $0\leq\lambda^*=\lambda^{**}\leq\tilde{\lambda}$.
In fact, for each $\lambda>\tilde{\lambda}$,
\begin{equation}
\lambda\IR\n{\vl}^rK\dx>\frac{r}{p}\IR\n{\gradvl}^p\dx+\frac{r}{q}\IR\n{\vl}^qV\dx\nonumber
\end{equation}
follows with some $\vl\in\MVR$ by homogeneity; this can be rewritten as
\begin{equation}
\Jc(\vl)=\frac{1}{p}\IR\n{\gradvl}^p\dx+\frac{1}{q}\IR\n{\vl}^qV\dx-\frac{\lambda}{r}\IR\n{\vl}^rK\dx<0,\nonumber
\end{equation}
which along with lemma \ref{L3.2} leads to $\Jc(\ul)=\inf\limits_{u\in\MVR}\Jc(u)\leq\Jc(\vl)<0$ for an $\ul\geq0$ in $\MVR$ that is a nontrivial positive solution to problem \eqref{eq1.2}, seeing $\Jc(\n{\ul})=\Jc(\ul)$.
So, one has $\lambda^{**}\leq\tilde{\lambda}$.
On the other hand, if $\lambda^*>\lambda^{**}$, one would find a $\lambda'\in\left[\lambda^{**},\lambda^*\right)$ such that problem \eqref{eq1.2} has at least a nontrivial positive solution at $\lambda'$ according to the definition of $\lambda^{**}$ - this however is against the definition of $\lambda^*$; if $\lambda^*<\lambda^{**}$, one would find a $\lambda'\in\left(\lambda^*,\lambda^{**}\right]$ such that problem \eqref{eq1.2} has at least a nontrivial positive solution at some $\mu'\left(<\lambda'\right)$ according to the definition of $\lambda^*$ - this however is against the definition of $\lambda^{**}$.
So, one has $\lambda^*=\lambda^{**}$.

Write $\lambda_1:=\lambda^*=\lambda^{**}$ in the sequel.

\begin{prop}\label{P3.3}
Under our {\bf Standing Assumptions (i)-(iii)}, one has $\lambda_1\geq0$ and problem \eqref{eq1.2} has a nontrivial positive solution $\ul\geq0$ in $\MVR$ for all $\lambda>\lambda_1$.
\end{prop}

\begin{proof}
By definition, $\lambda_1=\lambda^*$ so that if $\ul$ is a nontrivial positive solution to equation \eqref{eq1.2} in $\MVR$, then $\lambda\geq\lambda_1$.
Below, we verify \eqref{eq1.2} has at least one nontrivial solution $\ul\geq0$ in $\MVR$ for each $\lambda>\lambda_1$ using Struwe \cite[theorem 2.4]{St}; see also \cite[theorem 4.2]{AP}.

By definition, $\lambda_1=\lambda^{**}$; so, there is a $\mu\in\left[\lambda_1,\lambda\right)$ such that \eqref{eq1.2} has a nontrivial solution $\um\geq0$ in $\MVR$, which clearly is a sub-solution for \eqref{eq1.2} at $\lambda$.
Consider the constrained minimization problem $\inf\limits_{u\in\Ms}\Jc(u)$ with $\Ms:=\bre{u\in\MVR:u\geq\um}$.
Notice $\Ms$ is closed and convex, and thus is weakly closed in $\MVR$.
So, lemma \ref{L3.2} ensures the attainment of a minimizer of $\Jc$ in $\Ms$; that is, there is an $\ul\left(\geq\um\right)$ in $\Ms$ satisfying $\Jc(\ul)=\inf\limits_{u\in\Ms}\Jc(u)$.
Take $\varphi\in\CCON$, and set $\vv:=\max\bre{0,\um-\ul+\varepsilon\varphi}\geq0$ and $\ve:=\vv+\ul-\varepsilon\varphi\left(\geq\um\right)$ in $\Ms$ for some $\varepsilon>0$.
Then, one has $\Jc'(\ul)(\ul)\leq\Jc'(\ul)(\ve)$ that further implies
\begin{equation}\label{eq3.11}
\Jc'(\ul)(\varphi)\leq\frac{1}{\varepsilon}\Jc'(\ul)(\vv).
\end{equation}

Put $\Ov:=\bre{x\in\RN:\vv(x)>0}=\bre{x\in\RN:\ul(x)-\varepsilon\varphi(x)<\um(x)}\subseteq\operatorname{supp}(\varphi^+)$.
Since $\um$ is a sub-solution for \eqref{eq1.2} at $\lambda$ and $\vv\geq0$, $\Jc'(\um)(\vv)\leq0$ follows and one has
\begin{equation}
\begin{split}
&\,\Jc'(\ul)(\vv)\leq\Jc'(\ul)(\vv)-\Jc'(\um)(\vv)\\
\leq&-\IOv\left(\n{\gradul}^{p-2}\gradul-\n{\gradum}^{p-2}\gradum\right)\cdot\left(\gradul-\gradum\right)\dx\\
&+\varepsilon\,\Biggl\{\IOv\left(\n{\gradul}^{p-2}\gradul-\n{\gradum}^{p-2}\gradum\right)\cdot\nabla\varphi\dx\\
&\hspace{10mm}+\IOv\left(\ul^{q-1}-\um^{q-1}\right)\n{\varphi}V\dx+\lambda\IOv\left(\ul^{r-1}-\um^{r-1}\right)\n{\varphi}K\dx\Biggr\}\\
\leq&\,\,\varepsilon\,\Bigl\{\nm{\n{\nabla\varphi}}_{p,\hspace{0.2mm}\Ov}\Bigl[\nm{\n{\gradul}}^{p-1}_{p,\hspace{0.2mm}\Ov}+\nm{\n{\gradum}}^{p-1}_{p,\hspace{0.2mm}\Ov}\Bigr]
+\nm{\varphi}_{\LVOv}\Bigl[\nm{\ul}^{q-1}_{\LVOv}+\nm{\um}^{q-1}_{\LVOv}\Bigr]\\
&\hspace{6mm}+\lambda\nm{\varphi}_{\LKOv}\Bigl[\nm{\ul}^{r-1}_{\LKOv}+\nm{\um}^{r-1}_{\LKOv}\Bigr]\Bigr\}=o(\varepsilon)\nonumber
\end{split}
\end{equation}
as $\varepsilon\to0^+$, noticing $0<\vv\leq\varepsilon\n{\varphi}$ on $\Ov$.
This combined with \eqref{eq3.11} yields $\Jc'(\ul)(\varphi)\leq0$ for any $\varphi\in\CCON$ so that it implies $\Jc'(\ul)(-\varphi)\leq0$ as well.
By density, $\Jc'(\ul)(v)=0$ for all $v\in\MVR$, and thus $\ul\left(\geq\um\geq0\right)$ is a nontrivial solution to \eqref{eq1.2} at $\lambda$.
\end{proof}

Proposition \ref{P3.3} obviously provides the proof for the first assertion of theorem \ref{T1.1}.
In order to proceed as well as for convenience of the reader, we recall \cite[lemma 3.4]{Ha2}.

\begin{lem}\label{L3.4}
Let $\Om$ be a domain in $\RN$ for $N\geq1$, and let $f,g$ be two functions in $\LtOm$ for $t\in\left(1,\infty\right)$.
Then, there is a constant $C_t>0$, depending on $\Om,N,t$, such that
\begin{equation}
\IOm\left(\n{f}^{t-2}f-\n{g}^{t-2}g\right)\left(f-g\right)\dx\geq\left\{\begin{array}{ll}
C_t\nm{f-g}^t_{t,\hspace{0.2mm}\Om}&when\hspace{2mm}t\geq2,\\\\
C_t\,\frac{\nm{f-g}^2_{t,\hspace{0.2mm}\Om}}{\left(\nm{f}_{t,\hspace{0.2mm}\Om}+\nm{g}_{t,\hspace{0.2mm}\Om}\right)^{2-t}}&when\hspace{2mm}1<t<2.\nonumber
\end{array}\right.
\end{equation}
\end{lem}

\begin{lem}\label{L3.5}
Under the {\bf Standing Assumptions (i)-(iv)}, there are some absolute constants $C_K,\widehat{C}_{KV}>0$ such that $\lambda_1\geq\widehat{C}_{KV}>0$ and such that each nontrivial solution $\ul$ to equation $\eqref{eq1.2}$ in $\MVR$ satisfies
\begin{equation}\label{eq3.12}
\IR\n{\ul}^rK\dx\geq\left(\lambda\hspace{0.2mm}C^{\frac{p}{r}}_K\right)^{\frac{r}{p-r}}.
\end{equation}
\end{lem}

\begin{proof}
First, notice that at present we only consider $1<p<N$ and $p<r<\min\left\{p^*,q\right\}$.
Take $u\in\MVR$ to see, by \textsf{H\"{o}lder}'s \textsf{inequality} and \eqref{eq2.2} for an absolute constant $C_K>0$,
\begin{equation}\label{eq3.13}
\begin{split}
&\,\IR\n{u}^rK\dx\leq\left(\IR\n{u}^{p^*}\dx\right)^{\frac{r}{p^*}}\left(\IR K^{\frac{p^*}{p^*-r}}\dx\right)^{\frac{p^*-r}{p^*}}\\
\leq&\,\,C_K\left(\IR\n{\gradu}^p\dx\right)^{\frac{r}{p}}.
\end{split}
\end{equation}
Combining this with \eqref{eq3.3}, one observes, for each nontrivial solution $\ul$ to \eqref{eq1.2},
\begin{equation}
\left(\IR\n{\ul}^rK\dx\right)^{\frac{p}{r}}\leq\lambda\hspace{0.2mm}C^{\frac{p}{r}}_K\IR\n{\ul}^rK\dx,\nonumber
\end{equation}
which in turn yields \eqref{eq3.12}.
Moreover, by \eqref{eq3.2} and \eqref{eq3.12}, one sees $\left(\lambda\hspace{0.2mm}C^{\frac{p}{r}}_K\right)^{\frac{p}{p-r}}\leq\lambda^{\gamma}C^{\frac{p}{r}}_KC_{KV}$, and thus $\lambda\geq\widehat{C}_{KV}>0$ since $p<r$, which in particular implies $\lambda_1\geq\widehat{C}_{KV}>0$.
\end{proof}

Finally, we are ready to discuss the second twofold assertion of theorem \ref{T1.1}.

\begin{prop}\label{P3.6}
Under our {\bf Standing Assumptions (i)-(iv)}, one has $\lambda_1>0$ and problem \eqref{eq1.2} has a nontrivial positive solution $\ul\geq0$ in $\MVR$ if and only if $\lambda\geq\lambda_1$; besides, for $\lambda_2:=\tilde{\lambda}\left(\geq\lambda_1\right)$ as described in \eqref{eq3.10}, problem \eqref{eq1.2} has at least two distinct nontrivial positive solutions $\ul,\ult\geq0$ in $\MVR$ for every $\lambda>\lambda_2$.
\end{prop}

\begin{proof}
Recall when $\ul\geq0$ is a nontrivial positive solution to \eqref{eq1.2} in $\MVR$, then $\lambda\geq\lambda_1$; also, $\lambda_1>0$ was proved.
Hence, one only needs to show problem \eqref{eq1.2} has a nontrivial positive solution at $\lambda_1$.
Let $\bre{\lambda(k)>\lambda_1:k\geq1}$ decrease to $\lambda_1$, with $\bre{\ulk\geq0:k\geq1}$ an associated sequence of nontrivial solutions to \eqref{eq1.2} that is bounded in $\MVR$ by \eqref{eq3.2}.
Then, there is a subsequence $\bre{\ulk:k\geq1}$, using the same index $\lambda(k)$, with $\n{\gradulk}\rightharpoonup\n{\gradw}$ in $\LpR$, $\ulk\rightharpoonup\omega$ in $\LVR$ while $\ulk\to\omega$ in $\LKR$ for a function $\omega\in\MVR$ such that $\ulk\to\omega$ \textsl{a.e.} on $\RN$.
Then, one has \eqref{eq3.9} for $\ulk,\omega\geq0$ (instead of $u_k,u$) and
\begin{equation}
\lim_{k\to\infty}\IR\ulk^{r-1}vK\dx=\IR\omega^{r-1}vK\dx,\hspace{6mm}\forall\hspace{2mm}v\in\MVR;\nonumber
\end{equation}
see for example \cite[lemma 3.4]{AP}.
Keep in mind $\Jclk'(\ulk)=0$; that is,
\begin{equation}\label{eq3.14}
\IR\n{\gradulk}^{p-2}\gradulk\cdot\gradv\dx+\IR\ulk^{q-1}vV\dx=\lambda(k)\IR\ulk^{r-1}vK\dx
\end{equation}
for all $v\in\MVR$.
As $\Jc'(\omega)$ is a continuous, linear functional on $\MVR$,
\begin{equation}
\begin{split}
0\leftarrow&\,\,\Jclk'(\ulk)\left(\ulk-\omega\right)+\lambda(k)\IR\ulk^{r-1}\left(\ulk-\omega\right)K\dx\\
&-\Jc'(\omega)\left(\ulk-\omega\right)-\lambda\IR\omega^{r-1}\left(\ulk-\omega\right)K\dx\\
=&\IR\left(\n{\gradulk}^{p-2}\gradulk-\n{\gradw}^{p-2}\gradw\right)\cdot\left(\gradulk-\gradw\right)\dx\\
&+\IR\left(\ulk^{q-1}-\omega^{q-1}\right)\left(\ulk-\omega\right)V\dx\nonumber
\end{split}
\end{equation}
follows when $k\to\infty$, from which along with lemma \ref{L3.4} one derives $\ulk\to\omega$ in $\MVR$.
This in particular implies $\gradulk\to\gradw$ \textsl{a.e.} on $\RN$, so that we also have
\begin{equation}
\lim_{k\to\infty}\IR\n{\gradulk}^{p-2}\gradulk\cdot\gradv\dx=\IR\n{\gradw}^{p-2}\gradw\cdot\gradv\dx\nonumber
\end{equation}
and
\begin{equation}
\lim_{k\to\infty}\IR\ulk^{q-1}vV\dx=\IR\omega^{q-1}vV\dx\nonumber
\end{equation}
for all $v\in\MVR$.
Upon letting $k\to\infty$ on both sides of \eqref{eq3.14}, it shows
\begin{equation}\label{eq3.15}
\IR\n{\gradw}^{p-2}\gradw\cdot\gradv\dx+\IR\omega^{q-1}vV\dx=\lambda_1\IR\omega^{r-1}vK\dx.
\end{equation}
Using \cite[theorem 1.9]{LL} again and \eqref{eq3.12}, one sees $\IR\omega^rK\dx\geq\left(\lambda(1)\hspace{0.2mm}C^{\frac{p}{r}}_K\right)^{\frac{r}{p-r}}>0$ in view of the compact embedding $\MVR\hookrightarrow\LKR$.
So, $\omega\geq0$ is a nontrivial positive solution in $\MVR$ to equation \eqref{eq1.2}; that is, $\omega=u_{\lambda_1}$ in common practice notation.

On the other hand, from the discussions in \cite[lemma 3]{PR} and \cite[lemma 3]{Ra}, one knows any other solution $\ult\geq0$ to \eqref{eq1.2} at $\lambda\left(>\lambda_2\right)$, if it exists, should satisfy $\ult\leq\ul$ with $\Jc(\ul)<0$.
Furthermore, it is easily seen from \eqref{eq3.13} that
\begin{equation}\label{eq3.16}
\begin{split}
&\,\Jc(u)\geq\frac{1}{p}\IR\n{\gradu}^p\dx+\frac{1}{q}\IR\n{u}^qV\dx-\frac{\lambda}{r}\,C_K\left(\IR\n{\gradu}^p\dx\right)^{\frac{r}{p}}\\
\geq&\left(\frac{1}{p}-\frac{\lambda}{r}\,C_K\nm{\n{\gradu}}^{r-p}_{p,\hspace{0.2mm}\RN}\right)\nm{\n{\gradu}}^p_{p,\hspace{0.2mm}\RN}\geq\eta>0,
\end{split}
\end{equation}
provided $0<\nm{\n{\gradu}}_{p,\hspace{0.2mm}\RN}<\min\bre{\nm{\n{\gradul}}_{p,\hspace{0.2mm}\RN},\left(\frac{r}{\lambda pC_K}\right)^{\frac{1}{r-p}}}$.
Thus, exploiting the important mountain pass theorem of Candela and Palmieri \cite[theorem 2.5]{CP} (see also \cite[theorem a.3]{AP} for a closely related result using a different proof), there exists a sequence $\bre{u_l:l\geq1}$ in $\MVR$ (without loss of generality, we select $u_l\geq0$ as $\Jc(u)=\Jc(\n{u})$) satisfying $\Jc(u_l)\to c>0$ and $\Jc'(u_l)\to0$ in the dual space of $\MVR$ when $l\to\infty$, where $c:=\inf\limits_{h\in\mathscr{H}}\max\limits_{z\in\left[0,1\right]}\Jc(h(z))>0$ for $\mathscr{H}:=\bre{h\in C\left(\left[0,1\right];\MVR\right):h(0)=0,\,h(1)=\ul}$.
Lemma \ref{L3.2} yields a subsequence $\bre{u_l:l\geq1}$, using the same notation, such that $\n{\gradu_l}\rightharpoonup\n{\gradx}$ in $\LpR$, $u_l\rightharpoonup\xi$ in $\LVR$ and $u_l\to\xi$ in $\LKR$ for some function $\xi\in\MVR$ with $u_l\to\xi$ \textsl{a.e.} on $\RN$.
The same analyses then deduce $u_l\to\xi$ in $\MVR$.
Consequently, $\ult:=\xi\geq0$ is another nontrivial solution to \eqref{eq1.2} in $\MVR$ that is distinct from $\ul$ since $\Jc(\ult)=c>0$.
\end{proof}

% +++ below is final remark %%%%%%%%%%%%%%%%%%%%%%%%%%%%%%%%%%%%%%%%%%%%%%%%%%%%%%%%%%%%%%%%%%%%%%%%%%%%%%%%%%%%%%%%%%%%%%%%%%%%%%%%%%%%%%%%%%%%%%%%
% &&& %%%%%%%%%%%%%%%%%%%%%%%%%%%%%%%%%%%%%%%%%%%%%%%%%%%%%%%%%%%%%%%%%%%%%%%%%%%%%%%%%%%%%%%%%%%%%%%%%%%%%%%%%%%%%%%%%%%%%%%%%%%%%%%%%%%%%%%%%%%%%%
\vskip 2pt \small{\textbf{Final Note.} A careful reading of our proofs for propositions \ref{P2.3}, \ref{P2.4} and theorem \ref{T2.6} reveals the respective condition $K^{\beta_1}(x)/V^{\beta_2}(x)\in\LlR$ can be released to the weaker one: $K^{\beta_1}(x)/V^{\beta_2}(x)$ is eventually integrable as $\n{x}\to\infty$ for suitable exponents $\beta_1,\beta_2>0$ independent of $u$.}

% +++ below is appendix %%%%%%%%%%%%%%%%%%%%%%%%%%%%%%%%%%%%%%%%%%%%%%%%%%%%%%%%%%%%%%%%%%%%%%%%%%%%%%%%%%%%%%%%%%%%%%%%%%%%%%%%%%%%%%%%%%%%%%%%%%%%
% &&& %%%%%%%%%%%%%%%%%%%%%%%%%%%%%%%%%%%%%%%%%%%%%%%%%%%%%%%%%%%%%%%%%%%%%%%%%%%%%%%%%%%%%%%%%%%%%%%%%%%%%%%%%%%%%%%%%%%%%%%%%%%%%%%%%%%%%%%%%%%%%%
\appendix
\section{}
\noindent We provide below two compact embedding results that may be of independent interest; the proofs are omitted since they follow verbatim via lemma \ref{L2.2} and \cite[theorems 4.4 and 4.5]{Ha3}.

Recall a function $u\in\LlocR$ is said to \textsl{vanish at infinity} provided $\mathscr{L}\left(\left\{x\in\RN:\n{u(x)}\geq c\right\}\right)<\infty$ and \textsl{vanish at infinity weakly} provided $\lim\limits_{\n{x}\to\infty}\mathscr{L}\left(\mathbf{B}(x)\cap\left\{x\in\RN:\n{u(x)}\geq c\right\}\right)=0$ for all constants $c>0$, with $\mathscr{L}$ the Lebesgue measure and $\mathbf{B}(x)$ the unit ball centered at $x\in\RN$.

\begin{thm}\label{TA.1}
Suppose $N\leq p<\infty$, $1\leq q\leq r<\infty$, and $K(x),V(x)>0$ satisfy $K(x)\in\LaOm$ for some $\alpha\in\left(1,\infty\right]$ and all $\Om$ with $\mathscr{L}\left(\Om\right)<\infty$, $\inf\limits_{\RN}V(x)\geq V_0>0$ while $K(x)V^{-\tau}(x)$ vanishing at infinity with $\tau\in\left(0,1\right)$ if $q<r$ and $\tau=1$ if $q=r$.
Then, the embedding $\MVR\hookrightarrow\LKR$ is compact.
\end{thm}

\begin{thm}\label{TA.2}
Assume that $N\leq p<\infty$, $1\leq q\leq N$, $q\leq r<\infty$, and $K(x),V(x)>0$ satisfy $K(x)\in\LaR$ for some $\alpha\in\left(1,\infty\right]$, $\inf\limits_{\RN}V(x)\geq V_0>0$ whereas $K(x)V^{-\tau}(x)$ vanishing at infinity weakly with $\tau\in\left(0,1\right)$ if $q<r$ and $\tau=1$ if $q=r$.
Then, the embedding $\MVR\hookrightarrow\LKR$ is compact.
\end{thm}

% -----------------------------------------------------------------------------

\bibliographystyle{amsplain}
%\bibliography{xbib}

\begin{thebibliography}{21}

%
 \bibitem{AT} S. Alama \and G. Tarantello. Elliptic problems with nonlinearities indefinite in sign.
 {\sl J. Funct. Anal.} {\bf141} (1996), 159-215.

%
 \bibitem{ABC} A. Ambrosetti, H. Br\'{e}zis \and G. Cerami. Combined effects of concave and convex nonlinearities in some elliptic problems.
 {\sl J. Funct. Anal.} {\bf122} (1994), 519-543.

%
 \bibitem{AGP} A. Ambrosetti, J. Garcia-Azorero \and I. Peral. Multiplicity results for some nonlinear elliptic equations.
 {\sl J. Funct. Anal.} {\bf137} (1996), 219-242.

%
 \bibitem{AP} G. Autuori \and P. Pucci. Existence of entire solutions for a class of quasilinear elliptic equations.
 {\sl NoDEA Nonlinear Differential Equations Appl.} {\bf20} (2013), 977-1009.

%
 \bibitem{BR} L. Brasco \and B. Ruffini. Compact Sobolev embeddings and torsion functions.
 {\sl Ann. Inst. H. Poincar\'{e} Anal. Non Lin\'{e}aire} {\bf34} (2017), 817-843.

%
 \bibitem{CP} A.M. Candela \and G. Palmieri. Infinitely many solutions of some nonlinear variational equations.
 {\sl Calc. Var. Partial Differential Equations} {\bf34} (2009), 495-530.

%
 \bibitem{Ch} J. Chabrowski. Elliptic variational problems with indefinite nonlinearities.
 {\sl Topol. Methods Nonlinear Anal.} {\bf9} (1997), 221-231.

%
 \bibitem{CD} J. Chabrowski and J.M. do \'{O}. On semilinear elliptic equations involving concave and convex nonlinearities.
 {\sl Math. Nachr.} {\bf233/234} (2002), 55-76.

%
 \bibitem{Chi} R. Chiappinelli. Compact embeddings of some weighted Sobolev spaces on $\mathbb{R}^N$.
 {\sl Math. Proc. Cambridge Philos. Soc.} {\bf117} (1995), 333-338.

%
 \bibitem{GR} M. Ghergu \and V. R\u{a}dulescu.
 {\bf Nonlinear PDEs}. Springer, Heidelberg, 2012.

%
 \bibitem{GM} A. Gierer \and H. Meinhardt. A theory of biological pattern formation.
 {\sl Kybernetik (Berlin)} {\bf12} (1972), 30-39.

%
 \bibitem{Ha1} Q. Han. Positive solutions of elliptic problems involving both critical Sobolev nonlinearities on exterior regions.
 {\sl Monatsh. Math.} {\bf176} (2015), 107-141.

%
 \bibitem{Ha2} Q. Han. Compact embedding results of Sobolev spaces and positive solutions to an elliptic equation.
 {\sl Proc. Roy. Soc. Edinburgh Sect. A} {\bf146} (2016), 693-721.

%
 \bibitem{Ha3} Q. Han. Compact embedding results of Sobolev spaces and existence of positive solutions to quasilinear equations.
 {\sl Bull. Sci. Math.} {\bf141} (2017), 46-71.

%
 \bibitem{KS} E. Keller \and L. Segel. Initiation of slime mold aggregation viewed as an instability.
 {\sl J. Theoret. Biol.} {\bf26} (1970), 399-415.

%
 \bibitem{LL} E. Lieb \and M. Loss.
 {\bf Analysis}. American Mathematical Society, Providence, RI, 2001.

%
 \bibitem{Ma} V.G. Maz'ya.
 {\bf Sobolev Spaces}. Springer, Heidelberg, 2011.

%
 \bibitem{PR} P. Pucci \and V. R\u{a}dulescu. Combined effects in quasilinear elliptic problems with lack of compactness.
 {\sl Atti Accad. Naz. Lincei Rend. Lincei Mat. Appl.} {\bf22} (2011), 189-205.

%
 \bibitem{Rab} P. Rabier. Embeddings of weighted Sobolev spaces and generalized Caffarelli-Kohn-Nirenberg inequalities.
 {\sl J. Anal. Math.} {\bf118} (2012), 251-296.

%
 \bibitem{Ra} V. R\u{a}dulescu. Combined effects for a stationary problem with indefinite nonlinearities and lack of compactness.
 {\sl Dynam. Systems Appl.} {\bf22} (2013), 371-384.

%
 \bibitem{RS} V. R\u{a}dulescu and I. St\u{a}ncu\c{t}. Combined concave-convex effects in anisotropic elliptic equations with variable exponent.
 {\sl NoDEA Nonlinear Differential Equations Appl.} {\bf22} (2015), 391-410.

%
 \bibitem{Sc} M. Schneider. Compact embeddings and indefinite semilinear elliptic problems.
 {\sl Nonlinear Anal.} {\bf51} (2002), 283-303.

%
 \bibitem{St} M. Struwe.
 {\bf Variational Methods}. Springer-Verlag, Berlin, 2008.

%
 \bibitem{Wo} J.S.W. Wong. On the generalized Lane-Emden-Fowler equation.
 {\sl SIAM Rev.} {\bf17} (1975), 339-360.

\end{thebibliography}

\end{document}